\documentclass[12pt,leqno]{amsart}
\usepackage{amssymb,latexsym}
\usepackage{verbatim}   
\theoremstyle{plain}
\newtheorem{theorem}{Theorem}

\newtheorem{lemma}[theorem]{Lemma}

\theoremstyle{definition}

\theoremstyle{remark}

\begin{document}

\newcommand{\Aut}{\textup{Aut}}
\newcommand{\al}{\alpha}
\newcommand{\Char}{\textup{char}}
\newcommand{\D}{\Delta}
\newcommand{\Deg}{\textup{Deg}}
\newcommand{\diag}{\textup{diag}}
\newcommand{\E}{\mathbb{E}}
\newcommand{\e}{\varepsilon}
\newcommand{\End}{\textup{End}}
\newcommand{\F}{\mathbb{F}}
\newcommand{\FG}{{\F G}}
\newcommand{\FH}{{\F H}}
\newcommand{\G}{\Gamma}
\newcommand{\g}{\gamma}
\newcommand{\Gal}{\textup{Gal}}
\newcommand\GL{\textup{GL}}
\newcommand{\K}{\mathbb{K}}
\newcommand\lpower{{\sqsupset}}
\newcommand{\im}{\textup{im}}
\newcommand{\ind}{{\uparrow}}
\newcommand{\Ind}{\textup{Ind}}
\newcommand{\Mat}{\textup{Mat}}
\newcommand\PGL{\textup{PGL}}
\newcommand{\Q}{\mathbb{Q}}
\newcommand{\res}{{\downarrow}}
\newcommand\rpower{{\sqsubset}}
\newcommand{\s}{\sigma}
\newcommand\tensor{\otimes}
\newcommand{\Z}{\mathbb{Z}}

\hyphenation{afforded}
\hyphenation{induced}

\title[\tiny\upshape\rmfamily Modules induced from a normal subgroup
of prime index]{}
\date{Submitted: 30 October, 2002}

\begin{center}\large\sffamily\mdseries
   Modules induced from a normal subgroup of prime index
\end{center}

\author{{\sffamily S.\,P. Glasby}}

\begin{abstract}
Let $G$ be a finite group and $H$ a normal subgroup of
prime index $p$. Let $V$ be an irreducible $\F H$-module
and $U$ a quotient of the induced $\F G$-module $V\ind$.
We describe the structure of $U$, which is semisimple
when $\Char(\F)\ne p$ and uniserial if $\Char(\F)=p$. 
Furthermore, we describe the division rings arising as
endomorphism algebras of the simple components of $U$.
We use techniques from noncommutative ring theory to
study $\End_{\F G}(V\ind)$ and relate the right
ideal structure of $\End_{\F G}(V\ind)$ to the submodule
structure of $V\ind$.
\end{abstract}

\maketitle
\centerline{\noindent 2000 Mathematics subject classification: 
         20C40, 16S35}

\section{Introduction}\label{S:intro}

\noindent Throughout this paper $G$ will denote a finite group 
and $H$ will denote a normal subgroup of prime index
$p$. Furthermore, $V$ will denote an irreducible
(right) $\F H$-module, and $V\ind = V\tensor_{\F H} \F G$
is the associated induced $\F G$-module. Let $a$ be
an element of $G$ not in $H$, and let $\D:=\End_{\F H}(V)$
and $\G:=\End_{\F G}(V\ind)$.

This paper is motivated by the following problem:
``Given an irreducible $\F H$-module $V$, where
$\F$ is an arbitrary field, and a quotient $U$ of
$V\ind$, determine the submodule structure of $U$
and the endomorphism algebras of the simple modules.''
By Schur's lemma, $\D$ is a division algebra over $\F$,
so we shall need techniques from noncommutative ring
theory.

We determine the submodule structure of $U$ by explicitly
realizing $\End_\FG(U)$ as a direct sum of minimal right
ideals, or as a local ring. It suffices to solve our problem
in the case when $U=V\ind$. Henceforth $U=V\ind$.

In the case when $\F$ is algebraically closed of
characteristic zero, it is well known that two
cases arise. Either $V$ is $G$-stable and $V\ind$
is irreducible, or $V$ is not $G$-stable and
$V\ind$ is a direct sum of $p$ pairwise nonisomorphic
irreducible submodules. In [GK96] the structure of
$V\ind$ is analyzed in the case when $\F$ is an
arbitrary field satisfying $\Char(\F)\ne 0$.
The assumption that $\Char(\F)\ne 0$ was made
to ensure that $\D$ is a field. The main theorem
of [GK96] states that the structure of $V\ind$ is 
divided into five cases when $V$ is $G$-stable.
In this paper, we drop the hypothesis that
$\Char(\F)\ne 0$, and even more cases arise in
the stable case (Theorems 5, 8 and 9). Fortunately,
all these cases can be unified by considering
the factorization of a certain binomial $t^p-\lambda$
in a twisted polynomial ring $\D[t;\alpha]$, which
is a (left and right) principal ideal domain.

As we will focus on the case when $\F$ need not be
algebraically closed, a crucial role will be
played by the endomorphism algebra $\D=\End_{\F H}(V)$.
In [GK96] the submodules of $V\ind$ are described 
\emph{up to isomorphism}. As this paper is motivated
by computational applications we will strive towards
a higher standard: an explicit description of the
vectors in the submodule, and an explicit description
of the matrices in the endomorphism algebra of
the submodule. This is easily achieved in the
non-stable case, which we describe for the sake
of completeness.

\section{The non-stable case}

\noindent Let $e_0,e_1,\dots,e_{d-1}$ be an $\F$-basis for $V$
and let $\s\colon H\to \GL(V)$ be the
representation afforded by the irreducible
$\F H$-module $V$ relative to this basis.
The \emph{$g$-conjugate} of $\s$ ($g\in G$)
is the representation $g\mapsto (ghg^{-1})\s$,
and we say that $\s$ is \emph{$G$-stable} if
for each $g\in G$, $\s$ is equivalent to its
$g$-conjugate. In this section we shall assume that
$\s$ is \emph{not} $G$-stable.

Let $\s\ind\colon G\to \GL(V\ind)$ be the
representation afforded by $V\ind$ relative to the basis
\[
e_0,\dots,e_{d-1},e_0a,\dots,e_{d-1}a,\dots,
e_0a^{p-1},\dots,e_{d-1}a^{p-1}.
\]
Note that $G/H=\langle aH\rangle$ has order $p$, and we are
writing $e_ia^j$ rather than $e_i\tensor a^j$. Then
\[
a\s\ind=\begin{pmatrix}
0&I& &0\\
 & &\ddots& \\
0&0& &I\\
a^p\s&0& &0
\end{pmatrix},\qquad
h\s\ind=\begin{pmatrix}
h\s& & & \\
 &({}^ah)\s& & \\
 & &\ddots& \\
 & & &({}^{a^{p-1}}h)\s\end{pmatrix}
\]
where $h\in H$ and ${}^{a^i}h=a^i h a^{-i}$.
The elements of $\End_{\F G}(V\ind)$
are the matrices commuting with $G\s\ind$, namely
the $p\times p$ block scalar matrices
${\rm diag}(\delta,\dots,\delta)$ where $\delta\in\D$.

We shall henceforth assume that $V$ is $G$-stable.
In particular, assume that we know $\alpha\in\Aut_\F(V)$
satisfying
\[\tag{1}
(aha^{-1})\s=\alpha (h\s) \alpha^{-1}\qquad\text{for all $h\in H$.}
\]
There are `practical' methods for computing $\alpha$.
A crude method involves viewing $(ah_i a^{-1})\s\alpha=\alpha(h_i\s)$,
where $H=\langle h_1,\dots,h_r\rangle$, as a
system of $(d/e)^2r$ homogeneous linear equations over $\D$ in
$(d/e)^2$ unknowns where $e=|\D:\F|$. The solution space is $1$-dimensional
if $V$ is $G$-stable, and $0$-dimensional otherwise.
A more sophisticated method, especially when
$\Char(\F)\ne 0$, involves using the meat-axe algorithm, see
[P84], [HR94], [IL00], and [NP95]. There is also a recursive method
for finding $\alpha$ which we shall not discuss here.

\section{The elements of $\End_{\F G}(V\ind)$}

\noindent In this section we explicitly describe the matrices
in $\G= \End_{\F G}(V\ind)$ and give an
isomorphism $\G\to (\D,\alpha,\lambda)$ where
$(\D,\alpha,\lambda)$ is a general cyclic $\F$-algebra, see [L91,~14.5].
It is worth recalling that if $\Char(\F)\ne 0$, then
the endomorphism algebra $\D$ is commutative.
Even in this case, though, $\G$ can be noncommutative.

\begin{lemma}
If $i\in\Z$, then $\alpha^{-i}a^i\colon V\to Va^i$ is an 
$\F H$-isomorphism between the submodules $V$ and
$Va^i$ of $V\ind\res$.
\end{lemma}

\begin{proof}
It follows from Eqn (1) that 
\[
va^iha^{-i}=v\alpha^ih\alpha^{-i}\qquad\text{($i\in\Z$, $v\in V$)}
\]
Replacing $v$ by $v\alpha^{-i}$ gives $v\alpha^{-i}a^ih=vh\alpha^{-i}a^i$.
Hence $\alpha^{-i}a^i$ is an $\F H$-homomorphism,
and since it is invertible, it is an $\F H$-isomorphism.
\end{proof}

Conjugation by $\alpha$ induces an automorphism of $\D$,
which we also call~$\alpha$. [Proof: Conjugating the equation
$h\delta=\delta h$ by $\alpha$ and using (1) shows that $\alpha^{-1}\delta\alpha\in\D$.]
We abbreviate $\alpha^{-1}\delta\alpha=\delta^{\alpha}$ by $\alpha(\delta)$. The
reader can determine from the context whether the
symbol $\alpha$ refers to an element of $\Aut_\F(V)$,
or $\Aut_\F(\D)$.

It follows from Eqn (1) that 
\[
(a^p h a^{-p})\s = \alpha^p h\s \alpha^{-p}
\]
for all $h\in H$. Hence $\alpha^{-p}(a^p\s)$ centralizes $H$.
Therefore $\lambda:=\alpha^{-p}(a^p\s)$ lies in $\D^\times$.
Setting $h=a^p$ in Eqn (1) shows $a^p\s=\alpha(a^p\s)\alpha^{-1}$.
Thus 
\[
\alpha(\lambda)=\lambda^\alpha=(\alpha^{-p}(a^p\s))^\alpha
 =\alpha^{-p}(a^p\s)=\lambda.
\]
Conjugating by $\alpha^p=(a^p\s)\lambda^{-1}$ induces
an inner automorphism:
\[
\alpha^p(\delta)=\delta^{\alpha^p}=\delta^{(a^p\s)\lambda^{-1}}
=\delta^{\lambda^{-1}}=\lambda\delta\lambda^{-1}\qquad\text{($\delta\in\D$).}
\]
In summary, we have proved

\begin{lemma}
The element $\alpha\in\Aut_\F V$ satisfying \textup{Eqn (1)} induces via
conjugation an automorphism of $\D=\End_{\F H}(V)$, also called
$\alpha$. There exists $\lambda\in\D^\times$ satisfying
\[
\alpha^{-p}(a^p\s)=\lambda,\quad \alpha(\lambda)=\lambda,\quad
 \alpha^p(\delta)=\lambda\delta\lambda^{-1} \tag{2a,b,c}
\]
for all $\delta\in\D$.
\end{lemma}
\goodbreak

\begin{theorem}
The representation $\s\ind\colon G\to \GL(V\ind)$ afforded by
$V\ind$ relative to the $\F$-basis
\begin{align*}
  e_0,&e_1,\dots,e_{d-1},\dots,
    e_0\alpha^{-i}a^i,e_1\alpha^{-i}a^i,\dots,e_{d-1}\alpha^{-i}a^i,\\
  &\dots ,e_0\alpha^{-(p-1)}a^{p-1},e_1\alpha^{-(p-1)}a^{p-1},
    \dots,e_{d-1}\alpha^{-(p-1)}a^{p-1}\tag{3}
\end{align*}
for $V\ind$ is given by
\[
a\s\ind=\begin{pmatrix}
0&\alpha& &0\\
 & &\ddots& \\
0&0& &\alpha\\
\alpha\lambda&0& &0
\end{pmatrix},\quad
h\s\ind=\begin{pmatrix}
h\s& & & \\
 &h\s& & \\
 & &\ddots& \\
 & & &h\s\end{pmatrix}\tag{4a,b}
\]
where $h\in H$. Moreover, there is an isomorphism from the general cyclic
algebra $(\D,\alpha,\lambda)$ to $\End_{\F G}(V\ind)$
\[
(\D,\alpha,\lambda)\to \End_{\FG}(V\ind)\colon \sum_{i=0}^{p-1}\delta_ix^i
\mapsto \sum_{i=0}^{p-1} D(\delta_i)X^i
\]
where
\[
X=\begin{pmatrix}
0&I& &0\\
 & &\ddots& \\
0&0& &I\\
\lambda&0& &0
\end{pmatrix},\quad
D(\delta)=\begin{pmatrix}
\delta& & & \\
 &\alpha(\delta)& & \\
 & &\ddots& \\
 & & &\alpha^{p-1}(\delta)
\end{pmatrix}\tag{5a,b}
\]
and $\delta\in\D$.
\end{theorem}

\begin{proof}
By Lemma 1, $\alpha^{-i}a^i\colon V\to Va^i$ is an $\F H$-isomorphism.
Hence $h\s\ind$ is the $p\times p$ block scalar matrix
given by Eqn (4b). Similarly, $a\s\ind$ is given by Eqn (4a) as
\[
(v\alpha^{-i}a^i)a=v\alpha\alpha^{-(i+1)}a^{i+1}\quad\text{and}\quad
(v\alpha^{-(p-1)}a^{p-1})a=v\alpha(\alpha^{-p}a^p)=v\alpha\lambda
\]
where the last step follows from Eqn (2a).

We follow [J96] and write $R=\D[t;\alpha]$ for the twisted
polynomial ring with the usual addition, and multiplication
determined by $t\delta=\alpha(\delta)t$ for $\delta\in\D$. The right ideal
$(t^p-\lambda)R$ is two-sided as
\[
t(t^p-\lambda)=(t^p-\lambda)t\quad\text{and}\quad \delta(t^p-\lambda)
  =(t^p-\lambda)\lambda^{-1}\delta\lambda
\]
by virtue of Eqns (2b) and (2c). The general cyclic algebra
$(\D,\alpha,\lambda)$ is defined to be the quotient ring
$R/(t^p-\lambda)R$. Since $R$ is a (left) euclidean domain,
the elements of $(\D,\alpha,\lambda)$ may be written uniquely as
$\sum_{i=0}^{p-1} \delta_ix^i$ where $x=t+(t^p-\lambda)R$, and
multiplication is determined by the rules $x^p=\lambda$ and
$x\delta=\alpha(\delta)x$ where $\delta\in\D$.

The matrices commuting with $H\s\ind$ are precisely the
block matrices $(\delta_{i,j})_{0\le i,j<p}$ where $\delta_{i,j}
\in\D$. To compute $\End_{\F G}(V\ind)$, we determine
the matrices $(\delta_{i,j})$ that commute with $a\s\ind$. If $0<i<p-1$,
then comparing row $i-1$ of both sides of 
$(a\s\ind)(\delta_{i,j})=(\delta_{i,j})(a\s\ind)$
shows how to express $\delta_{i,j}$ in terms of the
$\delta_{i-1,k}$. Similarly, row $p-1$ shows how to express
$\delta_{p-1,j}$ in terms of $\delta_{0,k}$. It follows that
a matrix $(\delta_{i,j})$ commuting with $a\s\ind$ is completely
determined once we know the $0$th row $\delta_{0,j}$. We show
that the $0$th row can be arbitrary.  The $0$th row of
$\sum_{i=0}^{p-1}D(\delta_i)X^i$ is
$(\delta_0,\delta_1,\dots,\delta_{p-1})$.
This element lies in $\G$ as we show that $D(\delta_i),X\in\G$.

To see that $D(\delta)\in\G$, we show that
$(a\s\ind)D(\delta)=D(\delta)(a\s\ind)$. The first product equals
\[
(a\s\ind)D(\delta)=\begin{pmatrix}
0&\delta\alpha& &0\\
 & &\ddots& \\
0&0& &\alpha^{-(p-2)}\delta\alpha^{p-1}\\
\alpha\lambda\delta&0& &0
\end{pmatrix}
\]
and the second product is identical if
$\alpha\lambda\delta=\delta^{\alpha^{p-1}}\alpha\lambda$. However, this
is true by Eqn (2c). To see that
$X\in\G$,
write $a\s\ind=AX$ where $A=\diag(\alpha,\dots,\alpha)$.
It follows from Eqn (2b) that $A$ and $X$ commute. Therefore
$a\s\ind=AX$ and $X$ commute.

In summary, elements of $\G$ may be written
uniquely as $\sum_{i=0}^{p-1}D(\delta_i)X^i$ where
$\delta_i\in\D$. Since $X^p=\lambda I$ and $XD(\delta)=D(\alpha(\delta))X$
it follows that the map $\sum_{i=0}^{p-1}\delta_i x^i\mapsto
\sum_{i=0}^{p-1}D(\delta_i)X^i$  is an isomorphism
$(\D,\alpha,\lambda)\to\G$ as claimed.
\end{proof}

A consequence of Eqn (2c) is that $\alpha$ has order $p$ or
$1$ modulo the inner automorphisms of $\D$. It follows
from the Skolem-Noether theorem [CR90,3.62] that the order of
$\alpha$ modulo inner automorphisms is precisely the
order of the restriction $\alpha|Z$, where $Z=Z(\D)$ is the
centre of $\D$.

\section{The case when $\alpha|Z$ has order $p$}

\noindent In this section we determine the structure of
$\G:=\End_{\F G}(V\ind)$ in the case when
$\alpha$ induces an automorphism of order $p$ on the
field $Z(\D)$.

Of primary interest to us is Part (a) of the
following classical theorem. Although this result
can be deduced from [J96, Theorem~1.1.22] and the
fact that $t^p-\lambda$ is a `two-sided maximal' element
of $\D[t;\alpha]$, we prefer to give an elementary proof
which generalizes [L91, Theorem~14.6].

\begin{theorem}
Let $\G$ be the general cyclic algebra $(\D,\alpha,\lambda)$
where $\lambda\ne 0$ and $\alpha(\lambda)=\lambda$.
Suppose that $\alpha|Z(\D)$ has order $p$, and 
fixed subfield $Z_0$. Then\newline
\textup{(a)} $\G$ is a simple $Z_0$-algebra,\newline
\textup{(b)} $C_\G(\D)=Z(\D)$,\newline
\textup{(c)} $Z(\G)=Z_0$, and\newline
\textup{(d)} $|\G:Z_0|=(p\Deg(\D))^2$ where $|\D:Z(\D)|=\Deg(\D)^2$.
\end{theorem}

\begin{proof}
The following proof does not assume that $p$ is prime.
Let $\g=\g_1x^{i_1}+\cdots+\g_rx^{i_r}$ be a nonzero
element of an ideal $I$ of $\G$, where $0\le i_1<\cdots<i_r<p$,
$\g_i\in\D$, and $r$ is chosen minimal. By minimality,
each $\g_i$ is nonzero. To prove Part (a) it suffices to
to prove that $r=1$. Then $I=\G$ as $\g_1x^{i_1}\in I$
is a unit because $\g_1$ and $x$ are both units. Assume now that
$r>1$. Then
\[
(\g_1\alpha^{i_1}(\delta)\g_1^{-1})\g-\g\delta=
\sum_{k=2}^r \left(\g_1\alpha^{i_1}(\delta)\g_1^{-1}\g_k
-\g_k\alpha^{i_k}(\delta)\right)x^{i_k}
\]
lies in $I$ for each $\delta\in\D$. By the minimality of $r$,
each coefficient of $x^{i_k}$ is zero. This implies
that $\alpha^{i_1}$ equals $\alpha^{i_k}$ modulo inner automorphisms
for $k=2,\dots,r$. This contradiction proves Part (a).

The proofs of Parts (b) and (c) are  straightforward, so
we shall omit their proofs. Part (d) follows from $|\G:\D|=|Z(\D):Z_0|=p$,
and $|\D:Z(\D)|=\Deg(\D)^2$ is a square.
\end{proof}

Before proceeding to Theorem~5, we define the \emph{left-}
and \emph{right-twisted powers},
$\mu^{\lpower i}$ and $\mu^{i\rpower}$, where $\mu\in\D$ and $i\in\Z$.
These expressions are like norms, indeed Jacobson [J96] uses the notation
$N_i(\mu)$ to suggest this. These ``norms'', however, are not multiplicative
in general. Consider the twisted polynomial ring $\D[t;\alpha]$ and define
\[
  (\mu t)^i=\mu^{\lpower i}t^i,\quad\textup{and}\quad
  (t\mu)^i=t^i\mu^{i\rpower}
\]
for $\mu\in\D$ and $i\in\Z$. It follows from the power laws
$(\mu t)^i(\mu t)^j=(\mu t)^{i+j}$ and $((\mu t)^i)^j=(\mu t)^{ij}$ that
\[
  \mu^{\lpower i}\alpha^i(\mu^{\lpower j})=\mu^{\lpower (i+j)},
  \quad\textup{and}\quad
  \mu^{\lpower i}\alpha^i(\mu^{\lpower i})\cdots
    \alpha^{i(j-1)}(\mu^{\lpower i})=\mu^{\lpower (ij)}
\]
for $i,j\in\Z$. Similar laws hold for right-twisted powers. The
left-twisted powers of nonnegative integers can be defined by the
recurrence relation
\[
  \mu^{\lpower 0}=1,\quad\textup{and}\quad
  \mu^{\lpower (i+1)}=\mu^{\lpower i}\alpha^i(\mu)=\mu\alpha(\mu^{\lpower i})
  \quad\textup{for $i\ge 0$},\tag{6}
\]
and negative powers can be defined by
$\mu^{\lpower -i}=\alpha^i(\mu^{\lpower i})^{-1}$.

It is important in the sequel whether or not $\lambda^{-1}$ has a
left-twisted $p$th root.

\begin{theorem}
Let $V$ be a $G$-stable irreducible $\F H$-module where
$H\triangleleft G$ and $|G/H|=p$ is prime.
Let $\alpha$ and $\lambda$ be as in Lemma~2.
Suppose that $\alpha|Z$ has order $p$ where $Z=Z(\D)$ and
$\D=\End_{\FG}(V)$.\newline
\textup{(a)} If the equation $\mu^{\lpower p}=\lambda^{-1}$
has no solution for $\mu\in\D^\times$, then $V\ind$ is
irreducible, and $\End_{\F G}(V\ind)$ is isomorphic to
the general cyclic algebra $(\D,\alpha,\lambda)$ as per Theorem~3.\newline
\textup{(b)} If $\mu\in\D^\times$ satisfies $\mu^{\lpower p}=\lambda^{-1}$, then
$V\ind = U(\mu_0)\dotplus\cdots\dotplus U(\mu_{p-1})$ where
\[
  U(\mu_j)=V\sum_{i=0}^{p-1} \mu_j^{\lpower i}\alpha^{-i}a^i
  \qquad\qquad(j=0,1,\dots,p-1)
\]
are isomorphic irreducible submodules satisfying $U(\mu_j)\res\cong V$, and
where $\mu_j^{\lpower p}=\lambda^{-1}$. Moreover, if
$\rho\colon G\to\GL(U(\mu))$ is the representation afforded by $U(\mu)$
relative to the basis $e'_0,\dots,e'_{p-1}$ where 
\[
  e'_j=e_j\sum_{i=0}^{p-1} \mu^{\lpower i}\alpha^{-i}a^i\qquad
  (j=0,1,\dots,d-1),
\]
then $a\rho=\alpha\mu^{-1}$, $h\rho=h\sigma$ for $h\in H$, and
\[
  \End_{\FG}(U(\mu))=C_\D(\alpha\mu^{-1})=
  \{\delta\in\D\mid \delta^\alpha=\delta^\mu\}.
\]
\end{theorem}

\begin{proof}
By Theorem 4(a), $(\D,\alpha,\lambda)$ is a simple ring.
In Part(a) more is true: $(\D,\alpha,\lambda)$ is a division ring
by [J96, Theorem~1.3.16]. By Theorem~3, $(\D,\alpha,\lambda)$ is isomorphic to
$\End_{\F G}(V\ind)$ and so we have proved that $V\ind$ is irreducible
as desired.

Consider Part (b). Let $s=\mu t$ be an element of the
twisted polynomial ring $\D[t;\alpha]$, then
$s^i=(\mu t)^i=\mu^{\lpower i}t^i$
and
\[
  s\delta=\mu t\delta=(\mu\alpha(\delta)\mu^{-1})\mu t
   =\mu\alpha(\delta)\mu^{-1}s.
\]
Therefore the map $\D[t;\alpha]\to\D[s;\alpha\mu^{-1}]\colon\kern-2pt
\sum_{i=0}^{p-1} \delta_it^i\mapsto
\sum_{i=0}^{p-1} \delta_i(\mu^{\lpower i})^{-1}s^i$ 
is an isomorphism. We are abusing notation here by identifying
$\alpha\mu^{-1}$ in $\Aut_\F(V)$ with $\delta\mapsto\delta^{\alpha\mu^{-1}}$
in $\Aut_\F(\D)$. 
If $y=\mu x$, then
\[
  y^p=(\mu x)^p=\mu^{\lpower p}\lambda=\lambda^{-1}\lambda=1.
\]
By taking quotients we get an isomorphism
$(\D,\alpha,\lambda)\to(\D,\alpha\mu^{-1},1)$
given by
\[
  \sum_{i=0}^{p-1} \delta_ix^i\mapsto
  \sum_{i=0}^{p-1} \delta_i(\mu^{\lpower i})^{-1}y^i
\]
where $x=t+(t^p-\lambda)$ and
$y=s+(s^p-1)$.

As $y^p-1=(y-1)(y^{p-1}+\cdots+y+1)$, and $\D[s,\alpha\mu^{-1}]$
is right euclidean it follows that $(y-1)(\D,\alpha\mu^{-1},1)$ is a maximal
right ideal of $(\D,\alpha\mu^{-1},1)$. Now $y-1$ corresponds to
$\mu x-1$ which corresponds to $D(\mu) X -1$ whose kernel gives rise to
the irreducible submodule $U(\mu)$ of $V\ind$ in the statement of Part(b).
We shall reprove this, and prove a little more, using a more elementary
argument.

Let $U$ be a submodule of $V\ind$ satisfying $U\res\cong V$.
Let $\phi\colon V\to V\ind$ be an $\F H$-homomorphism such that
$V\phi=U\res$. Let $\pi_i\colon V\ind\to Va^i$ be the $\F H$-epimorphism
given by $(\sum_{i=0}^{p-1}v_i\alpha^{-i}a^i)\pi_i=v_i\alpha^{-i}a^i$.
Then $\delta_i=\phi\pi_i a^{-i}\alpha^i$ is an $\F H$-homomorphism
$V\to V$, or an element of $\D$. Since $\pi_0+\pi_1+\cdots+\pi_{p-1}$
is the identity map $1\colon V\ind\to V\ind$, it follows that
\[
  \phi=\phi 1=\phi(\pi_0+\pi_1+\cdots+\pi_{p-1})
  =\sum_{i=0}^{p-1}\delta_i\alpha^{-i}a^i.
\]
We now view $\phi$ as a map $V\to U$ and note that $U=Ua$.
Then $\alpha^{-1}a\colon V\to Va$, $a^{-1}\phi a\colon Va\to Ua$
and $\phi^{-1}\colon Ua\to V$ are each $\FH$-isomorphisms.
Hence their composite,
$(\alpha^{-1}a)(a^{-1}\phi a)\phi^{-1}$ is an isomorphism $V\to V$,
denoted $\mu^{-1}$ where $\mu\in\D^\times$.
Rearranging gives $\phi a=\alpha\mu^{-1}\phi$. Therefore,
\[
  (v\phi)a=\left(v\sum_{i=0}^{p-1}\delta_i\alpha^{-i}a^i\right)a=
  v\alpha\mu^{-1}\sum_{i=0}^{p-1}\delta_i\alpha^{-i}a^i
\]
for all $v\in V$. The expression $(v\delta_i\alpha^{-i}a^i)a$ equals
\[
  v\delta_i\alpha\alpha^{-(i+1)}a^{i+1}
  =v\alpha\delta_i^\alpha\alpha^{-(i+1)}a^{i+1}
  =v\alpha\mu^{-1}\delta_{i+1}\alpha^{-(i+1)}a^{i+1}.
\]
Setting $i=p-1$ gives
\[
  (v\delta_{p-1}\alpha^{-(p-1)}a^{p-1})a
  =v\alpha\delta_{p-1}^\alpha\alpha^{-p}a^p
  =v\alpha\delta_{p-1}^\alpha\lambda=v\alpha\mu^{-1}\delta_0.
\]
Therefore $\delta_i^\alpha=\mu^{-1}\delta_{i+1}$ for $i=0,\dots,p-2$ and
$\delta^\alpha_{p-1}\lambda=\mu^{-1}\delta_0$. If $\delta_0=0$, then
each $\delta_i=0$ and $\phi=0$, a contradiction. Thus $\delta_0\ne 0$ and
as $V\delta_0^{-1}\phi=U$,
we may assume that $\delta_0=1$. It follows from Eqn~(6)
that $\delta_i=\mu^{\lpower i}$ is the solution to
the recurrence relation: $\delta_0=1$ and
$\mu\delta_i^\alpha=\delta_{i+1}$ for $i\ge0$. Furthermore
$\mu\delta^\alpha_{p-1}=\lambda^{-1}$ implies
that $\mu^{\lpower p}=\lambda^{-1}$.
In summary, any submodule $U$ of $V\ind$ satisfying $U\res\cong V$
equals $U(\mu)$ for some $\mu$ satisfying $\mu^{\lpower p}=\lambda^{-1}$.
Furthermore, by retracing the above argument, if
$\mu^{\lpower p}=\lambda^{-1}$, then $U(\mu)$ is an irreducible submodule
of $V\ind$ satisfying $U\res\cong V$.

As $\End_{\FG}(V\ind)$ is a simple ring, $V\ind$ is a direct sum
of isomorphic simple submodules. Therefore, 
$V\ind = U(\mu_0)\dotplus\cdots\dotplus U(\mu_{p-1})$ as desired.
It follows from Lemma~1 that the representation $\rho\colon G\to\GL(V)$
satisfies $a\rho=\alpha\mu^{-1}$
and $h\rho=h\sigma$ for $h\in H$.
Consequently, the matrices commuting with $G\rho$ equal
the elements of $\D$ centralizing $a\rho$. Hence
$\End_{\F G}(U(\mu))=C_\D(\alpha\mu^{-1})$ as claimed.
\end{proof}

\section{The case when $\alpha$ is inner}

\noindent In this section assume that $\alpha|Z(\D)$ has order $1$,
or equivalently by the Skolem-Noether theorem,
that $\alpha$ is inner. Fix $\e\in\D^\times$ such that $\alpha$ is the 
inner automorphism $\alpha(\delta)=\e^{-1}\delta\e$. Clearly
$\alpha(\e)=\e$ and by Eqn~(2c)
$\e^{-p}\delta\e^p=\alpha^p(\delta)=\lambda\delta\lambda^{-1}$. Therefore,
$\eta=\e^p\lambda\in Z(\D)$.
If $y=\e x$, then $y^p=\e^{\lpower p}x^p=\e^p\lambda=\eta$
and $y\delta=\e x\delta=\e\delta^\e x=\delta\e x =\delta y$.
Hence
\[
(\D,\alpha,\lambda)\to(\D,1,\eta)\colon \sum_{i=0}^{p-1} \delta_ix^i\mapsto
 \sum_{i=0}^{p-1} \delta_i\e^{-i}y^i\tag{8}
\]
is an isomorphism. Thus we may untwist $\End_{\FG}(V\ind)$. 

\begin{theorem}
Let $V$ be a $G$-stable irreducible $\F H$-module where
$H\triangleleft G$ and $|G/H|=p$ is prime.
Suppose that $\alpha$ induces the inner
automorphism $\alpha(\delta)=\delta^\e$ of the division algebra
$\D=\End_{\F H}(V)$. Then
$\eta=\e^p\lambda\in Z^\times$ where $Z=Z(\D)$. Suppose that
$s^p-\eta=\nu(s)\mu(s)$
where $\mu(s)=\sum_{i=0}^m \mu_i s^i$ and $\nu(s)=\sum_{i=0}^{p-m} \nu_i s^i$,
are monic polynomials in $\D[s]$. Then
$W_\mu=\sum_{i=0}^{m-1}V\sum_{j=0}^{p-m}\nu_j\e^{i+j}\alpha^{-(i+j)}a^{i+j}$
is a submodule of $V\ind$. Let $\rho\colon G\to\GL(W_\mu)$ be the
representation afforded by $W_\mu$ relative to the basis
\[
  e'_0,\dots,e'_{d-1},\dots,e'_j(\e X)^k,\dots,
    e'_0(\e X)^{m-1},\dots,e'_{d-1}(\e X)^{m-1}\tag{9}
\]
where
\[
  e'_k=e_k\sum_{j=0}^{p-m}\nu_j\e^j\alpha^{-j}a^j=
  e_k\sum_{j=0}^{p-m}\nu_j(\e X)^j,
\]
and $X$ is given
by \textup{Eqn~(5a)}. Then
\[
  a\rho=\alpha\e^{-1}\begin{pmatrix}
0&1& &0\\
 & &\ddots& \\
0&0& &1\\
-\mu_0&-\mu_1& &-\mu_{m-1}
\end{pmatrix},\tag{10}
\]
and $h\rho=\diag(h\s,\dots,h\s)$ where $h\in H$. Moreover,
\[
  \End_{\FG}(W_\mu)=\left\{\sum_{i=0}^{m-1}\delta_i (a\rho)^i\mid
                   \delta_i\in\D\right\}.
\]
If $\mu(s)\in Z[s]$, then $\End_{\FG}(W_\mu)\cong \D[s]/\mu(s)\D[s]\cong
\D\tensor_Z \K$ where $\K=Z[s]/\mu(s) Z[s]$.
\end{theorem}

\begin{proof}
Arguing as in Theorem~5, we have a series of right ideals:
\def\hs{\hskip1.5mm}
\[
  \nu(s)\D[s]\subseteq\D[s],\hs\nu(y)(\D,1,\eta)\subseteq(\D,1,\eta),\hs
  \nu(\e x)(\D,\alpha,\lambda)\subseteq(\D,\alpha,\lambda),
\]
and $\sum_{i=0}^n D(\nu_i)(\e X)^i\G$ is a right ideal of
$\G=\End_\FG(V\ind)$.
This right ideal corresponds to the submodule
$V\ind\sum_{i=0}^n D(\nu_i)(\e X)^i\G$
of $V\ind$. It follows
from Eqn~(5a) and $(\e X)^p-\eta=0$
that the minimum polynomial of $\e X$ equals $s^p-\eta$.

Let $v'=v\nu(\e X)$ where $v\in V$. Then
\[
  v'\mu(\e X)=v\nu(\e X)\mu(\e X)=v((\e X)^p-\eta)=v\thinspace0=0.\tag{11}
\]
This proves that (9) is a basis for
\[
  W_\mu=\text{im}\;\nu(\e X)=\ker\nu(\e X)=
  \sum_{i=0}^{m-1}V\sum_{j=0}^n\nu_j\e^{i+j}\alpha^{-(i+j)}a^{i+j}
\]
It follows from Lemma~1 that $h\rho=\diag(h\s,\dots,h\s)$ is a block
scalar matrix ($h\in H$). Since $a=\alpha X$,
\[
  v'(\e X)^i a=v'(\e X)^i \alpha X=v'\alpha\e^{-1}(\e X)^{i+1}.\tag{12}
\]
It follows from Eqns~(11) and (12) that the matrix for $a\rho$ is correct.

It is now a simple matter to show that
$\left\{\sum_{i=0}^{m-1}\delta_i (a\rho)^i\mid\delta_i\in\D\right\}$
is contained in $\End_{\FG}(W_\mu)$.
A familiar calculation shows that an element of $\End_{\FG}(W_\mu)$
is determined by the entries in its top row. As this may be arbitrary,
we have found all the elements of $\End_{\FG}(W_\mu)$.
\end{proof}

It follows from Theorem~6 that a necessary condition for $W_\mu$ to
be irreducible is that $\mu(s)$ is irreducible in $\D[s]$.
Lemma~7 describes an important case when $\End_\FG(W_\mu)$ is a division
ring, and hence $W_\mu$ is irreducible.
The following proof follows Prof. Deitmar's suggestion [D02].

\begin{lemma}
Let $\D$ be a division algebra with center $\F$, and let $\mu(s)\in\F[s]$
be irreducible of prime degree.
Suppose that no $\delta\in\D$ satisfies $\mu(\delta)=0$. Then the quotient
ring $\D[s]/\mu(s)\D[s]$ is a division algebra.
\end{lemma}

\begin{proof} Let $\K=\F[s]/\mu(s)\F[s]$. Then $\K$ is a field and
$|\K:\F|=\deg\mu(s)$ is prime. Clearly $\mu(s)\D[s]$ is a two-sided ideal of 
$\D[s]$, and $\D[s]/\mu(s)\D[s]$ is isomorphic to $\D_\K=
\D\tensor_\F \K$.
By [L91, 15.1(3)], $\D_\K$ is a central simple $\K$-algebra, and hence
is isomorphic to $M_n(D)$ for some division algebra $D$ over $\F$.
The \emph{degree} of $D$ and the \emph{Schur index} of $\D_\K$ are
defined as follows
\[
  \Deg(D)=(\dim_\F D)^{1/2}\quad\text{and}\quad \Ind(\D_\K)=\Deg(D).
\]
By [P82, Prop. 13.4], $\Ind(\D_\K)$ divides $\Ind(\D)$, and
$\Ind(\D)$ divides $|\K:\F|\,\Ind(\D_\K)$. Thus either
\[
\Ind(\D_\K)=\Ind(\D)=\Deg(\D)=\Deg(D_\K)
\]
and $\D_\K$ is a division algebra by [P82, Prop. 13.4(ii)], or
$\Ind(\D)$ equals $|\K:\F|\,\Ind(\D_\K)$. If the second case occurred, then
by [P82, Cor. 13.4], $\K$ is isomorphic to a subfield of $\D$, and
so $\mu(s)$ has a root in $\D$, contrary to our hypothesis.
\end{proof}

If $\eta\not\in\D^p$, then $\eta\not\in Z^p$ and so $s^p-\eta$ is irreducible
in $Z[s]$, and it follows from Lemma~7 that $V\ind=W_{s^p-\eta}$ is
irreducible. Note that $\End_{\FG}(V\ind)\cong \D\tensor Z[\eta^{1/p}]$
is a division algebra.

\section{The case when $\alpha$ is inner and $\xi^p=\eta$}

\noindent In this section we shall assume that $\xi\in\D^\times$ satisfies
$\xi^p-\eta=0$. Let $y=\e x$ and $z=\xi^{-1} y=\xi^{-1}\e x$.
It is useful to consider the isomorphisms
$(\D,\alpha,\lambda)\to(\D,1,\eta)\to(\D,1,1)$ defined by
$x\mapsto\e^{-1}y$ and $y\mapsto\xi z$. Note $y$ and $z$ are central
in $(\D,1,\eta)$ and $(\D,1,1)$ respectively, and $y^p=\eta$ and $z^p=1$.

\begin{theorem}
Let $V$ be a $G$-stable irreducible $\F H$-module where
$H\triangleleft G$ and $|G/H|=p$ is prime.
Suppose that $\alpha$ induces the inner
automorphism $\alpha(\delta)=\delta^\e$ of the division algebra
$\D=\End_{\F H}(V)$. Set $\eta=\e^p\lambda$ and let
$\xi,\omega\in\D$ satisfy $\xi^p=\eta$ and $\omega^p=1$. Then $\xi\in Z=Z(\D)$.
\newline
\textup{(a)} If $\Char(\F)\ne p$ and $\omega\ne 1$, then $V\ind$ is an
internal direct sum
\[
  V\ind=U(\xi)\dotplus U(\xi\omega)\dotplus\cdots\dotplus U(\xi\omega^{p-1})
\]
where
\[
  U(\xi\omega^j)=V\sum_{i=0}^{p-1} (\xi\omega^j)^{-i}\e^i\alpha^{-i}a^i
\]
is irreducible, and $U(\xi\omega)\cong U(\xi\omega')$ if and only if
$\omega$ and $\omega'$ are conjugate in $\D$. If $\mu(s)$ is an
irreducible factor of $s^p-\eta$ in $Z[s]$, then $W_\mu$ defined in
Theorem~6 is a Wedderburn component of $V\ind$, and 
$W_\mu=U(\theta_1)\dotplus\cdots\dotplus U(\theta_n)$ where
$\theta_1,\dots,\theta_n$ are the roots of $\mu(s)$ in the field
$Z(\xi,\omega)$. In addition, the representation
$\rho_{\theta}\colon G\to\GL( U(\theta))$ afforded by $U(\theta)$
relative to the basis $e'_0,\dots,e'_{d-1}$ where
\[
  e'_j=e_j\sum_{i=0}^{p-1} \theta^{-i}\e^i\alpha^{-i}a^i
\]
satisfies
\[
  a\rho_{\theta}=\alpha\e^{-1}\theta\qquad\text{and}\qquad
  h\rho_{\theta}=h\sigma\tag{12a,b}
\]
for $h\in H$, and $\End_{\FG}(U(\theta))=C_\D(\theta)$.
\newline
\textup{(b)} If $\Char(\F)=p$, then $\omega=1$ and $V\ind$ is uniserial with
unique composition series $\{0\}=W_0\subset W_1\subset\cdots\subset W_p=V\ind$
where
\[
  W_k=\sum_{i=1}^k V\sum_{j=0}^{p-i}
       {i+j-1\choose j}\xi^{-j}\e^j\alpha^{-j}a^j.
\]
Moreover, $W_{k-1}/W_k\cong U(\xi)$ for $k=1,\dots,p$
and $\End_{\FG}(U(\xi))=\D$.
\end{theorem}

\begin{proof}
Since $z\delta=\delta z$, we see that
$(\xi^{-1}\e x)\delta=\delta(\xi^{-1}\e x)$. This implies that
$\xi^{-1}\delta=\delta\xi^{-1}$ and so $\xi\in Z$.\newline
Case (a): Now $(\xi\omega)^p=\xi^p\omega^p=\eta$, hence
\[
  y^p-\eta=y^p-(\xi\omega)^p=(y-\xi\omega)
  \left(\sum_{i=0}^{p-1}(\xi\omega)^{p-1-i}y^i\right).\tag{13}
\]
Therefore $V\ind\sum_{i=0}^{p-1}(\xi\omega)^{p-1-i}(\e X)^i\G$ is a
submodule of $V\ind$ where $X$ is given by Eqn~(5a).
We show directly that $U(\xi\omega)$
is a submodule of $V\ind$. This follows from
\begin{align*}
  (v(\xi\omega)^{-i}\e^i\alpha^{-i}a^i)a&=
  v\alpha(\alpha^{-1}(\xi\omega)^{-i}\e^i\alpha)\alpha^{-(i+1)}a^{i+1}\\
  &=v\alpha\e^{-1}\xi\omega(\xi\omega)^{-(i+1)}\e^{i+1}\alpha^{-(i+1)}a^{i+1}
  \tag{14}
\end{align*}
and setting $i=p-1$ in the right-hand side of Eqn~(14) gives
\[
  v\alpha\e^{-1}\xi\omega(\xi\omega)^{-p}\e^{p}\alpha^{-p}a^p
  =v\alpha\e^{-1}\xi\omega\eta^{-1}\e^{p}\lambda
  =v\alpha\e^{-1}\xi\omega.
\]
As $U(\xi\omega)\res\cong V$, we see that $U(\xi\omega)$ is an irreducible
$\FG$-submodule of $V\ind$. Setting $\theta=\xi\omega$ establishes the
truth of Eqns~(12a,b).

We may calculate $\textup{Hom}(U(\xi\omega),U(\xi\omega'))$ directly
by finding all $\delta$ in $\End_\F(V)$ that intertwine $\rho_{\xi\omega}$
and $\rho_{\xi\omega'}$. As $\delta$ intertwines $h\rho_{\xi\omega}$
and $h\rho_{\xi\omega'}$, it follows that $\delta$
commutes with $H\sigma$, and hence $\delta\in\D$. Also
\[
  \delta(\alpha\e^{-1}\xi\omega)=(\alpha\e^{-1}\xi\omega')\delta
\]
so $\delta^{\alpha\e^{-1}}\xi\omega=\xi\omega'\delta$. Since
$\xi\in Z^\times$ and $\delta^{\alpha\e^{-1}}=\delta$,
this amounts to $\delta\omega=\omega'\delta$.
Setting $i=j$ shows that
$\End_{\FG}(U(\xi\omega))=C_\D(\omega)$.

The Galois group $\textup{Gal}(Z(\omega):Z)$ is cyclic of order
dividing $p-1$.  Also $\omega$ and $\omega'$ are conjugate in
$\textup{Gal}(Z(\omega):Z)$ if and only if they share the same
minimal polynomial over $Z$. The latter holds by Dixon's Theorem
[L91, 16.8] if and only if $\omega$ and $\omega'$ are conjugate in $\D$.
Note that $\omega$ and $\omega'$ share the same minimal polynomial over
$Z$ precisely when  $\xi\omega$ and $\xi\omega'$ share the same minimal
polynomial. This proves that $W_\mu$ is a Wedderburn component of $V\ind$.

Case (b): Suppose now that $\Char(\F)=p$. Then $\omega=1$ and Eqn~(13) becomes
$y^p-\eta=(y-\xi)^p=(y-\xi)(\sum_{i=0}^{p-1}{p-1\choose i}(-\xi)^{p-1-i}y^i)$.
As
\[
  \G=\End_\FG(V\ind)\cong(\D,\alpha,\lambda)\cong(\D,1,\eta)\cong(\D,1,1)
  \cong\D[z]/(z-1)^p\D[z]
\]
has a unique composition series, so too does $V\ind$. By noting that
$z=\xi^{-1}\e x$ and
$D(\xi^{-1}\e)=\xi^{-1}\e$, we see that $W_i=V\ind(\xi^{-1}\e X-1)^{p-i}\G$
defines the unique composition series for $V\ind$ where $X$ is given
by Eqn~(5a).

Let $R$ be the diagonal matrix $\diag(1,\xi^{-1}\e,\dots,(\xi^{-1}\e)^{p-1})$,
and let $S$ be the matrix whose $(i,j)$th block is the binomial coefficient
${i\choose j}$ where $0\le i,j<p$. A direct calculation verifies that
$R(\xi^{-1}\e X)R^{-1}=C$ and $S^{-1}CS=J$ where
\[
C=\begin{pmatrix}0&1& &0\\
 & &\ddots& \\
0&0& &1\\
1&0& &0
\end{pmatrix}
\quad\text{and}\quad
J=\begin{pmatrix}
1&1& & \\
 &\ddots& & \\
  &   &1&1\\
  & & &1
\end{pmatrix}.
\]
Therefore $\xi^{-1}\e X-1=T^{-1}(J-1)T$ where $T=S^{-1}R$, and hence
\[
  W_k=V\ind(\xi^{-1}\e X-1)^{p-k}=V\ind\;T^{-1}(J-1)^{p-k}T=V\ind (J-1)^{p-k}T.
\]
It is easily seen that $\im (J-1)^{p-k}=\ker (J-1)^k$ is the subspace
$(0,\dots,0,V,\dots,V)$ where the first $V$ is in column $p-k$.
The $(i,j)$th entry of $T=S^{-1}R$ is $(-1)^{i+j}{i\choose j}(\xi^{-1}\e)^j$.
The last row of $T$ gives
\[
  W_1=V\sum_{j=0}^{p-1}(-1)^{p-1+j}{p-1\choose j}(\xi^{-1}\e)^j\alpha^{-j}a^j.
\]
More generally, the last $k$ rows of $T$ give
\[
  W_k=\sum_{i=1}^kV\sum_{j=0}^{p-i}
       (-1)^{p-i+j}{p-i\choose j}(\xi^{-1}\e)^j\alpha^{-j}a^j.
\]
Since $p-i-\ell=-(i+\ell)$ in a field (such as $\F$) of characteristic $p$,
we see that ${p-i\choose j}=(-1)^j{i+j-1\choose j}$ and the formula for
$W_k$ simplifies to
\[
  W_k=\sum_{i=1}^k V\sum_{j=0}^{p-i}{i+j-1\choose j}\xi^{-j}\e^j\alpha^{-j}a^j.
\]
Setting $k=1$ shows $W_1=U(\xi)$. A direct calculation
shows that $W_{i-1}/W_i\cong U(\xi)$. We showed in Part~(a) that
$\End_\FG(U(\xi))$ equals $C_\D(\xi)=\D$.
\end{proof}

In Case (a), $C_\D(\xi\omega)$ equals $\D$ precisely when $\omega\in Z$.
If $\D$ is
the rational quaternions, and $\omega$ is primitive cube root of unity, then
$C_\D(\omega)$ equals $\Q(\omega)$. There are infinitely many primitive cube
roots of~1 in this case, and they form a conjugacy class of $\D$
by Dixon's Theorem (as they all satisfy the irreducible polynomial $s^2+s+1$
over $\Q$). Thus isomorphism of the submodules $U(\xi\omega)$
is governed by conjugacy in $\D$, and not conjugacy in
$\textup{Gal}(\Q(\omega):\Q)$.

Finally, it remains to generalize Theorem~8(a) to allow for the possibility
that $\D$ may not contain a primitive $p$th root of 1.

\begin{theorem}
Let $V$ be a $G$-stable irreducible $\F H$-module where
$H\triangleleft G$ and $|G/H|=p$ is prime.
Suppose that $\e,\xi\in\D$ satisfy $\alpha(\delta)=\delta^\e$ ($\delta\in\D$)
and $\xi^p-\eta=0$ where $\eta=\e^p\lambda\in Z=Z(\D)$.
In addition, suppose that $\Char(\F)\ne p$.
Then $V\ind$ is an internal direct sum
\[
  V\ind=W_{\mu_1}\dotplus\cdots\dotplus W_{\mu_r}
\]
where $s^p-\eta=\mu_1(s)\cdots\mu_r(s)$ is a factorization into
monic irreducibles over $Z$, and where $W_\mu$ defined in
Theorem~6. If $\mu(s)$ is a monic irreducible factor of $s^p-\eta$,
and $\mu(s)=\nu_1(s)\cdots\nu_n(s)$ where the $\nu_i(s)$ are monic and
irreducible in $\D[s]$, then $W_\mu$ is a Wedderburn component of
$V\ind$, and $W_\mu\cong W_{\nu_n}^{\oplus n}$ where $W_{\nu_n}$
is an irreducible $\FG$-module and $\End_{\FG}(W_{\nu_n})$ is given
in Theorem~6. In addition,
\[
  \End_{\FG}(W_{\nu_n})\cong B/\nu_n(s)\D[s]
\]
where $B=\{\delta(s)\in\D[s]\mid \delta(s)\nu_n(s)\in\nu_n(s)\Delta[s]\}$ is
the idealizer of the right ideal $\nu_n(s)\Delta[s]$.
\end{theorem}

\begin{proof}
Since $\Char(\F)\ne p$, the monic polynomials $\mu_1(s),\dots,\mu_r(s)$
are distinct and pairwise coprime in $Z[s]$. From this it follows that
$V\ind$ equals $W_{\mu_1}\dotplus\cdots\dotplus W_{\mu_r}$. By Theorem~6,
$\End_\FG(W_\mu)\cong\D_\K$ where 
$\D_\K\cong\D[s]/\mu(s)\D[s]\cong\D\tensor_Z\K$, and $\K$ is
the field $Z[s]/\mu(s)Z[s]$. By [L91, 15.1(3)], $\D_\K$ is a
simple ring. Therefore $\mu(s)\D[s]$ is a two-sided maximal ideal of $\D[s]$,
and so $\mu(s)$ is called a two-sided maximal element of $\D[s]$.
By [J96, Theorem 1.2.19(b)], $\D_\K\cong M_n(D)$ where $D$ is the
division ring $B/\nu_n(s)\D[s]$. Moreover, $Z(\D_\K)\cong Z(M_n(D))$
so $\K\cong Z(D)$. Thus $W_\mu\cong W_{\nu_n}^{\oplus n}$ where
$W_{\nu_n}$ is an
irreducible submodule of $V\ind$ and $\End_\FG(W_{\nu_n})\cong D$.
In addition, $\nu_1,\dots,\nu_n$ are similar [J96, Def. 1.2.7],
and $W_{\nu_1},\dots,W_{\nu_n}$ are isomorphic.

If $\mu(s),\mu'(s)$ are distinct monic irreducible factors of
$s^p-\eta$ in $Z[s]$ and $\nu(s),\nu'(s)$ in $\D[s]$ are monic
irreducible factors of $\mu(s)$ and $\mu'(s)$ respectively, then it follows
from [J96, Def. 1.2.7] that $\nu(s)$ and $\nu'(s)$ are not similar. This
means that an
irreducible summand of $W_\mu$ is not isomorphic to an irreducible
summand of $W_{\mu'}$. Hence the $W_\mu$ are Wedderburn components as claimed.
\end{proof}

\section*{Acknowledgment}

\noindent I am very grateful to Prof. A.\,D.\,H. Deitmar for
providing a proof [D02] of of Lemma~7 in the case when $s^p-\eta$
has no root in $\D$.

\section*{References}

\begin{itemize}

\item[{[CR90]}] C.W.\,Curtis and I.\,Reiner, Methods of Representation
Theory: with Applications to Finite Groups and Orders, Vol. 1,
Classic Library Edn, John Wiley and Sons, 1990.

\item[{[D02]}] A.\,D.\,H. Deitmar, {\tt sci.math.research}, September 6, 2002.

\item[{[GK96]}] S.P.\,Glasby and L.G.\,Kov\'acs, \emph{Irreducible modules and
normal subgroups of prime index}, Comm. Algebra 24 (1996), no. 4,
1529--1546. (MR~97a:20012)

\item[{[IL00]}] G. Ivanyos and K. Lux, Treating the exceptional cases
of the MeatAxe, Experiment. Math. 9 (2000), no. 3, 373--381. (MR 2001j:16067)

\item[{[HR94]}] D.F.\,Holt and S.\,Rees, \emph{Testing modules for
irreducibility}, J. Austral. Math. Soc. Ser. A 57 (1994), no. 1,
1--16. (MR~95e:20023)

\item[{[J96]}] N. Jacobson, Finite-Dimensional Division Algebras over Fields,
Springer-Verlag, 1996.

\item[{[L91]}] T.Y.\,Lam, A First Course in Noncommutative Rings, Graduate
Texts in Mathematics 131, Springer-Verlag, 1991.

\item[{[NP95]}] P.M.\,Neumann and C.E.\,Praeger, \emph{Cyclic matrices over
finite fields}, J. London Math. Soc. 52 (1995), no. 2,
263--284. (MR~96j:15017)

\item[{[P84]}] R.A.\,Parker, The computer calculation of modular characters
(the meat-axe), Computational group theory (Durham, 1982), 267--274,
Academic Press, London, 1984. (MR~84k:20041)

\item[{[P82]}] R.\,S. Pierce, Associative Algebras, Graduate Texts in
Mathematics 88, Springer-Verlag, 1982.

\end{itemize}

\vskip3mm
\goodbreak
{\tiny\scshape
\begin{tabbing}
\=Central Washington University\=\kill
\>S.\,P. Glasby                \\
\>Department of Mathematics    \\
\>Central Washington University\\
\>WA 98926-7424, USA           \\
\>\scriptsize\upshape\ttfamily GlasbyS@cwu.edu\\
\end{tabbing}
}

\end{document}